\documentclass[12pt]{elsarticle}

\usepackage{amssymb,latexsym,amsmath,epsfig}
\usepackage{graphicx,relsize}
\usepackage{mathtools}
\usepackage{mathrsfs}               
\usepackage{amsthm}
\newtheorem{thm}{Theorem}
\newtheorem{prop}[thm]{Proposition}
\newtheorem{lem}[thm]{Lemma}
\newtheorem{defn}{Definition}

\newtheorem{cor}[thm]{Corollary}

\font\smallit=cmti10

\newcommand{\mise}{mis\`{e}re }
\newcommand{\by}{\times}
\newcommand{\R}{\mathcal{R}}
\renewcommand{\L}{\mathcal{L}}
\newcommand{\N}{\mathcal{N}}
\renewcommand{\P}{\mathcal{P}}

\newcommand{\Gb}{\overline{G}}

\renewcommand{\v}{\phantom{'} \vert \phantom{'}}

\newcommand{\mpl}{mis\`{e}re-play}

\newcommand{\lopt}{\boldsymbol{L}}
\newcommand{\ropt}{\boldsymbol{R}}
\newcommand{\bs}{\boldsymbol}

\begin{document}
\begin{center}
\uppercase{\bf Dead ends in mis\`ere play: the mis\`ere monoid of canonical numbers}
\vskip 20pt
{\bf Rebecca Milley}\\
{\smallit Dept.~of Mathematics and Statistics, Dalhousie University, Halifax, Canada, B3H 4R2}\\
{\tt rkmilley@mathstat.dal.ca}\\
\vskip 10pt
{\bf Gabriel Renault}\\
{\smallit Univ.~Bordeaux, LaBRI, UMR5800, F-33400 Talence, France\\
CNRS, LaBRI, UMR5800, F-33400 Talence, France}
\\
{\tt gabriel.renault@labri.fr}\\ 
\end{center}
\vskip 30pt

\section*{Abstract}
We find the \mise monoids  of normal-play canonical-form integer and non-integer numbers.  These come as consequences of more general results for the universe of {\em dead-ending} games.  Left and right {\em ends} have previously been defined as games in which Left or Right, respectively, have no moves; here we define a dead left (right) end to be a  left (right) end whose options are all left (right) ends, and we define a dead-ending game to be one in which all end followers are dead.  We find the monoids and partial orders of dead ends, integers, and all numbers, and construct an infinite family of games that are equivalent to zero in the dead-ending universe.

\section*{Keywords} 
Combinatorial game; Partizan; Mis\`ere; Monoids; Dead-ending.

\section{Introduction}
In many {\em combinatorial games} (two-player games of perfect information and no chance), players take turns placing pieces on a board according to some set of rules.  Usually these rules imply that the board spaces available to a player on his or her turn are a subset of those available on the previous turn; games such as {\sc domineering}, {\sc col}, {\sc snort}, {\sc hex}, and {\sc nogo}, among many others, fit this description.  What sorts of properties do these {\em placement games} have over other games?  What can be said about their game trees?  In contrast to a game like {\sc maze} or {\sc konane}, placement games have the property that a player cannot `open up' moves for him or herself, or for the opponent; in particular, if a player has no available moves at some position of the game then they will have no moves in any follower\footnote{By {\em follower} we mean any position that can be reached from a given game position, by alternating or non-alternating play, including the original position itself.} of that position.  This particular property, which we call {\em dead-ending}, is the focus of the present paper. 

A {\em left end} is a game position with no options for the player we call Left, and a {\em right end} is a position with no options for the player Right.  A game with no options for either player (called the {\em zero game}) is both a left end and a right end.  
  We can thus define {\em dead-ending} games as follows.

\begin{defn} A left (right) end is a {\em dead end} if every follower is also a left (right) end.  A game $G$ is called {\em dead-ending} if every end follower of $G$ is a dead end. The set of all dead-ending games is denoted $\mathcal{E}$.\end{defn}

In addition to the games listed above, many  well-studied positions from normal-play game theory have the dead-ending property: integers are dead ends, and non-integer numbers, all-small games, and all {\sc hackenbush} positions are dead-ending.  The set of all dead-ending games is thus a meaningful (and large) universe to consider. Restricted universes are of particular interest to those studying {\em mis\`ere} games (where the last player to move loses), since the restrictions may reintroduce some of the algebraic structure that is lost in general \mise play.  The dead-ending universe is a natural extension of the dicot\footnote{In a {\em dicot} game (called {\em all-small} in normal play), the position and every follower satisfies the property that either both players have a move or the game is over. These games are trivially dead-ending, as no follower can be a non-zero end.} universe, which has been the focus of recent research in \mise game theory (see \cite{Allen2009},\cite{Allen2011}, and \cite{McKayMN}, for example).

In this paper we establish some basic results for dead-ending games and demonstrate that several significant subuniverses (ends, integers, and non-integer numbers) have many of the `nice' algebraic properties that are missing from general mis\`ere play.  More specifically, we find the {\em mis\`ere monoids} of these sets of games, and determine the associated partial orders.  The concept of a \mise monoid, and other prerequisite material, is discussed in Section 1.1.

\subsection{General \mise background}
By convention, the players Left and Right are female and male, respectively.  Under {\em normal play}, the first player unable to move in a game loses; the less-studied and less-structured ending condition known as {\em \mise play} declares that the first player unable to move is the winner.
Games or {\em positions} are defined in terms of their options: $G=\{G^{\lopt} \v G^{\ropt} \}$, where 
$G^{\lopt}$ is the set of positions $G^L$ to which Left can move in one turn, and similarly for $G^{\ropt}$.
  The simplest game is the zero game, $0=\{\cdot \v \cdot\}$, where the dot indicates an empty set of options.

In both play conventions, the outcome classes {\em next} ($\N$), {\em previous} ($\P$), {\em left} ($\L$), and \textit{right} ($\R$) are partially ordered as shown in Figure \ref{NPLR},
with Left preferring moves toward the top and Right preferring moves toward the bottom.  We use $\N^-$, $\P^-$, $\L^-$, and $\R^-$ to denote the outcome classes under \mise play.
We also use the outcome functions $o^-(G)$ and $o^+(G)$ to distinguish between the \mise and normal outcomes, respectively.

\begin{figure}[htp]
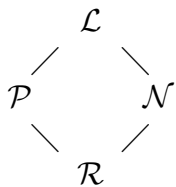

\begin{center}
\begin{tabular}{c}
$\L$\\
$\diagup \quad \quad \diagdown$\\
$\P \quad \quad \; \;	\quad \N$\\
$\diagdown \quad \quad \diagup$\\
$\R$\\
\end{tabular}
\end{center} 
\caption{The partial order of outcome classes.}
\label{NPLR}
\end{figure}
  
Many definitions from normal-play game theory\footnote{A complete overview of normal-play game theory can be found in \cite{AlberNW2007}.} are used without modification for \mise games, including birthday, disjunctive sum, equality, and inequality.  Thus, for \mise games, 
$$G= H \textrm{ if } o^-(G+X)= o^-(H+X) \textrm{ for all games } X,$$
$$G\geq H \textrm{ if } o^-(G+X)\geq o^-(H+X) \textrm{ for all games } X.$$
In normal-play, the \textit{negative} of a game is defined recursively as $-G=\{-{G}^{\ropt} | 
{-G}^{\lopt}\}$, and is so-called because $G+ (-G)=0$ for all games $G$.  Under general mis\`{e}re play, however, this property holds only if $G$ is the zero game \cite{MesdaO2007}.  To avoid confusion and inappropriate cancellation, we write $\overline{G}$ instead of $-G$ and refer to this game as the {\em conjugate} of $G$.

In normal-play games, there is an easy test for equality: $G=0$ if and only if $o^+(G)=\P$, and so $G=H$ if and only if $o^+(G-H)=\P$.
 In \mpl, no such test exists.  Equality of \mise games is difficult to prove and is, moreover, relatively rare: for example, besides $\{\cdot \v \cdot\}$ itself, there are no games equal to the zero game under \mise play \cite{MesdaO2007}.  As with normal-play, we can reduce a \mise game to a unique canonical form by eliminating dominated options and bypassing reversible ones \cite{siege}, but instances of domination (inequality) and reversibility are much less common under \mise play.
Plambeck (see \cite{Plamb2005} and \cite{PlambS2008}, for example) introduced a partial solution to these challenges: modify the definitions of equality and inequality by restricting the game universe. Specifically, given a set of games $\mathcal{U}$, {\em \mise equivalence (modulo} $\mathcal{U}${\em )} is defined by 
$$G\equiv H \textrm{ (mod } \mathcal{U} \textrm{) if } o^-(G+X)= o^-(H+X) \textrm{ for all games } X \in \mathcal{U},$$
while {\em \mise inequality (modulo} $\mathcal{U}${\em)} is defined by
$$G\geqq H \textrm{ (mod } \mathcal{U} \textrm{) if } o^-(G+X)\geq o^-(H+X) \textrm{ for all games } X \in \mathcal{U}.$$
We use the words {\em equivalent} and {\em indistinguishable} interchangeably, and if $G\not \equiv H$ (mod $\mathcal{U}$) then $G$ and $H$ are said to be {\em distinguishable}.  If $G\not \geqq H$ and $G\not \leqq H$ then $G$ and $H$ are {\em incomparable}, and we write $G||H$. In this paper we use the symbol $\gneqq$ to indicate strict modular inequality.   The set $\mathcal{U}$ is often called the {\em universe}. 

Given a universe $\mathcal{U}$, we can determine the equivalence classes under 
$\equiv (\mbox{mod }\mathcal{U})$ 
and form the quotient semi-group $\mathcal{U}/\equiv$.  This quotient, together with the tetra-partition of
elements into the sets $\P^-$, $\N^-$, $\R^-$, and $\L^-$, is called the {\em \mise monoid} of the set $\mathcal{U}$,  denoted $\mathscr{M}_{\mathcal{U}}$.  It is usually desirable to have the set of games $\mathcal{U}$ closed under disjunctive sum; when a set of games is not already thus closed, we consider its {\em closure} or set of all sums of those games.  

 Plambeck and Siegel \cite{PlambS2008} 
used the monoid approach with 
much success in analyzing impartial \mise games. Allen \cite{Allen2009,Allen2011} extended the idea
to partizan game theory, investigating monoids within the dicot
universe. Most recently, McKay, Milley, and Nowakowski \cite{McKayMN}
 determined the monoid for a subset of the dicot universe  corresponding to instances of {\sc hackenbush}, and Milley, Nowakowski, and Ottaway \cite{MilleyNO} found the monoid for ends  in the universe of alternating (or `consecutive-move-ban') games.  In this paper we make a significant contribution to the literature by determining the \mise monoid of all normal-play canonical-form numbers.
We find this to be the same as the monoid of the closure of dead ends.

In Section 2 we establish some basic properties of the dead-ending universe.  In Section 3 we analyze ends in the dead-ending universe, which includes all integers in normal-play canonical form.  In Section 4 we extend this analysis to non-integer numbers and find that the monoid of all numbers is equivalent to the monoid of integers.  We also determine the partial orders of these subuniverses (modulo the subuniverse as well as modulo $\mathcal{E}$), and establish invertibility of the elements (modulo $\mathcal{E}$).  Finally, in Section 5 we discuss other dead-ending games, in the context of equivalency to zero modulo the dead-ending universe.

\section{Preliminary results}
We begin with some immediate consequences of the definition of dead-ending.  Lemmas \ref{follower_closed} and \ref{sum_closed} show that the universe $\mathcal{E}$ of dead-ending games is `closed' in two important respects: it is closed under followers and closed under disjucntive sum.

\begin{lem} \label{follower_closed} If $G$ is dead-ending then every follower of $G$ is dead-ending.\end{lem}

\begin{proof} 
If $H$ is a follower of $G$, then every follower of $H$ is also a follower of $G$; thus if $G$ satisfies the definition of dead-ending, then so does $H$.
\end{proof}

\begin{lem} \label{sum_closed}If $G$ and $H$ are dead-ending then $G+H$ is dead-ending. \end{lem}
\begin{proof}
Any follower of $G+H$ is of the form $G'+H'$ where $G'$ and $H'$ are (not necessarily proper) followers of $G$ and $H$, respectively.  If $G'+H'$ is a left end, then both $G'$ and $H'$ are left ends, which must be dead, since $G$ and $H$ are dead-ending.  Thus, any right options ${G'}^R$ and ${H'}^R$ are left ends, and so all options ${G'}^R+H'$ and $G'+{H'}^R$ of $G'+H'$ are left ends.   A symmetric argument holds if $G'+H'$ is a right end, and so $G+H$ is dead-ending.
\end{proof}

Under \mise play,
Left trivially wins any left end playing first.  In general, Left may or may not win a left end playing second; for example, the game $\{\cdot \v 1\}$ is a left end in $\N^-$.  If a (non-zero) left end is dead, however, then it is a win for Left playing first or second.

\begin{lem} \label{leftendL} If $G\not = 0$ is a dead left end then $G\in \L^-$, and if $G\not =0 $ is a dead right end then $G\in \R^-$.
\end{lem}
\begin{proof} A left end is always in $\L^-$ or $\N^-$.  If $G$ is a dead left end then any right option $G^R$ is also a left end, so Right has no good first move.  Similarly, a dead right end is in $\R^-$.
\end{proof}

The following lemma, which describes a sufficiency condition for invertibility, applies generally to any mis\`ere universe.  In the present paper we apply it to various subsets of dead-ending games.

\begin{lem} \label{zerolemmaS}
Let $\mathcal{U}$ be any game universe closed under conjugation, and let $S\subseteq \mathcal{U}$ be a set of games closed under followers.  
If $G+\overline{G} + X\in \L^-\cup \N^-$ 
for every game $G \in S$ and every left end $X \in \mathcal{U}$, 
 then $G+\Gb \equiv 0$ (mod $\mathcal{U}$) for every $G\in S$.
\end{lem}

\begin{proof}
Let $S$ be a set of games with the given conditions. Since $\mathcal{U}$ is closed under conjugation,  by symmetry we also  have $G+\overline{G} + X\in \R^-\cup \N^-$ for every $G\in S$ and every {\em right} end $X\in \mathcal{U}$. 

Let $G$ be any game in $S$ and assume inductively that $H+\overline{H}\equiv 0$ (mod $\mathcal{U}$) for every follower $H$ of $G$. 
Let $K$ be any game in $\mathcal{U}$, and suppose Left wins $K$.  We must show that Left can win $G+\Gb+K$.  
Left should follow her usual strategy in $K$; if Right plays in $G$ or $\Gb$ to, say, $G^R+\Gb+K'$, with $K'\in \L^-\cup\P^-$, then Left copies his move and wins as the second player on $G^R+\Gb^L +K'= G^R + \overline{G^R}+K' \equiv 0+K'$, by induction.  Otherwise, once Left runs out of moves in $K$, say at a left end $K''$, she wins playing next on $G+\Gb+K''$ by assumption.  	
\end{proof}

The main argument above begins with the phrase `suppose Left wins $K$'. This appears to be ambiguous, or incomplete; does Left win $K$ playing first or playing second? The implied assumption with such a statement  is that the argument to follow holds for both cases.  

In subsequent  sections we refer to the two game functions below, which are well-defined for our purposes  --- namely, for numbers and ends. 

\begin{defn} The {\em left-length} of a game $G$, denoted $l(G)$, is the minimum number of consecutive left moves required for Left to reach zero in $G$.  The {\em right-length} $r(G)$ of $G$ is the minimum number of consecutive right moves required for Right to reach zero in $G$.
\end{defn}

 In general, left- and right-length are  well-defined if $G$ has a non-alternating path to zero for Left or Right, respectively, and if the shortest of such paths is never dominated by another option.  The latter condition ensures $l(G)=l(G')$ when $G\equiv G'$.  As suggested above, both of these conditions are met if $G$ is a (normal-play) canonical-form number or if $G$ is an end in $\mathcal{E}$.
If $l(G)$ and $l(H)$ are both well-defined then $l(G+H)$ is defined  and $l(G+H)=l(G)+l(H)$.  Similarly, when right-length is defined for $G$ and $H$, we have $r(G+H)=r(G)+r(H)$.

\section{Integers and other dead ends}

Let $\bs{n}$ denote the game $\{\bs{n-1} \v \cdot \}$, where $\bs{0}=0=\{\cdot \v \cdot\}$.  That is, $\bs{n}$ is the position with the same game tree as the integer $n$ in normal-play canonical form.  In this paper the term `integer' will always refer to such a position.  Note that we should distinguish between the game $\bs{n}$ and the number $n$, since, among other shortcomings, $\bs{n} \not > \bs{n-1}$ in general \mise play.  One property that does hold in both normal and mis\`ere play is that the disjunctive sum of positive integers $\bs{n}$ and $\bs{m}$ is the integer $\bs{n+m}$, although this is not generally true (in \mise games) if one of $n$ or $m$ is negative and the other positive.  In this section we prove that the restricted universe of integers under \mise play has much of the structure enjoyed by normal-play integers.

An integer is an example of a dead end: if $n>0$ then Right has no move in $\bs{n}$ and no move in any follower of $\bs{n}$.  Similarly, if $n<0$ then $\bs{n}$ is a dead left end.  Thus, the following results for ends in the dead-ending universe are true for all integers, modulo $\mathcal{E}$.

Our first result shows that when all games in a sum are dead ends, the outcome is completely determined by the left- and right-lengths of the games.  
\begin{lem}\label{end_outcome}
If $G$ is a dead right end and $H$ is a dead left end, then 
\begin{equation*}
o^-(G+H) = 
\begin{cases} 
\N^- & \text{if } l(G)=r(H), \\
\L^- & \text{if } l(G)<r(H), \\
\R^- & \text{if } l(G)>r(H). \\
\end{cases}
\end{equation*}
\end{lem}

\begin{proof} 
Each player has no choice but to play in their own game, and so the winner will be the player who can run out of moves first. \end{proof}

We use Lemma \ref{end_outcome} to prove the following theorem, which demonstrates the invertibility of all ends in $\mathcal{E}$.  In particular, this shows that every integer  has an additive inverse  modulo $\mathcal{E}$.  

\begin{thm}\label{endszero}
If $G$ is a dead end then $G+\Gb \equiv 0$ (mod $\mathcal{E}$).
\end{thm}

\begin{proof}
Assume without loss of generality that $G\not = 0$ is a dead right end. 
Since every follower of a dead end is also a dead end, Lemma \ref{zerolemmaS} applies, with $S$ the set of all dead left and right ends. It therefore suffices to show $G+\Gb+X\in \L^-\cup \N^-$ for any left end $X$ in $\mathcal{E}$.  We have $l(G)=r(\Gb)$ and $r(X)\geq 0$,  so $l(G)\leq r(\Gb) +r(X) = r(\Gb+X)$, which gives $G+\Gb+X\in L^-\cup \N^-$  by Lemma \ref{end_outcome}.
\end{proof}

\begin{cor} \label{int_zero}
If $n$ is an  integer then $\bs{n}+\overline{\bs{n}}\equiv 0$ (mod $\mathcal{E}$).
\end{cor}

Note that equivalency in $\mathcal{E}$  implies equivalency in all subuniverses of $\mathcal{E}$; thus in the universe of integers alone, every game has an inverse. 

Lemma \ref{end_outcome}  shows that when playing a sum of dead ends, both players aim to exhaust their own options as quickly as possible.  This suggests that options with longer paths to zero will be dominated by shorter paths; in particular, we have that integers are totally ordered among dead ends, as established in Theorem \ref{int_order} below. Note that this ordering only holds in the subuniverse of the closure\footnote{Since the sum of a dead left end and a dead right end may not be a dead end (or any end at all), the set of dead ends is not closed under disjunctive sum; thus the universe we consider is the {\em closure} of dead ends, as defined in Section 1.1.} of dead ends, and not in the whole universe $\mathcal{E}$.  In fact, we see immediately in Theorem \ref{intfuzzy} that distinct integers are pairwise incomparable modulo $\mathcal{E}$, just as they are in the general mis\`ere universe.

In the following arguments we frequently use the fact that, when $H\in \mathcal{U}$ has an additive inverse modulo $\mathcal{U}$, $G\gneqq H$ (mod $\mathcal{U}$) if and only if $G+\overline{H}\gneqq 0$.

\begin{thm}\label{int_order}
If $n<m\in \mathbb{Z}$ then $\bs{n} \gneqq \bs{m}$ modulo the closure of dead ends.
\end{thm}

\begin{proof}  
By Corollary \ref{int_zero}, it suffices to show $\bs{n}+\overline{\bs{m}}\gneqq 0$ (equivalently, $\bs{k}\gneqq 0$ for any negative integer $k$), modulo the closure of dead ends.  Let $X$ be any game in the closure of dead ends; then $X=Y+Z$ where $Y$ is a dead right end and $Z$ is a dead left end.  Suppose Left wins $X$ playing first; then by Lemma \ref{end_outcome}, $l(Y)\leq r(Z)$.  We need to show Left wins $\bs{k}+X$, so that $o^-(\bs{k}+X)\geq o^-(X)$.  Since $\bs{k}$ is a negative integer,  $r(\bs{k})$ is defined and $r(\bs{k})= -k >0$. ÊThus $l(Y)\leq r(Z) <r(Z)+r(\bs{k}) = r(Z+\bs{k})$, which gives $\bs{k} + Y +Z=\bs{k}+X\in \L^-\cup\N^-$, by Lemma \ref{end_outcome} .
\end{proof}

In general, $G \geq H$ under \mise play implies $G \geq H$ under normal play \cite{siege}; Theorem 8 shows this is not always the case for \mise inequality modulo  a restricted universe.  

\begin{thm}\label{intfuzzy}
If $n\not = m \in \mathbb{Z}$ then $\bs{n} || \bs{m}$ (mod ${\mathcal{E}}$).
\end{thm}

\begin{proof} Assume $n>m$.  Then we have $\bs{n} \not \geqq \bs{m}$ (mod $\mathcal{E}$), because $\bs{n}+\overline{\bs{m}}\in \R^-$ while $\bs{m}+\overline{\bs{m}}\equiv 0 \in \N^-$.  It remains to show $\bs{n} \not \leqq \bs{m}$.

Define a family of games $\lambda_k$ by 
$$\lambda_1 = \{0 \v -\bs{1}\}, \lambda_k = \{0 \v \lambda_{k-1}\}.$$
Note that $\bs{n}+\lambda_n \in \L^-$, since Left wins playing first or second by ignoring $\lambda_n$ and forcing Right to play there, bringing the game to $-\bs{1}$ with either Left or Right to play next.

If $n>m\geq0$ then $\bs{m}+\lambda_n$ is in $\P^-$ or $\R^-$: Left loses as soon as she plays in $\lambda_n$, and so plays only in $\bs{m}$, but (moving first) she will run out of moves in $\bs{m}$ before $\lambda_n$ is brought to $-1$.  Thus $\bs{n} \not \leqq \bs{m}$ in this case, since Left can win $\bs{n}+\lambda_n$ but not $\bs{m}+\lambda_n$.

If $m<0$ then let $k = -m-1$ and take $X=\bs{k} +\lambda_{n+k}$.  As above, $\bs{n}+\bs{k}+\lambda_{n+k} \in \L^-$.  However, $\bs{m}+\bs{k}+\lambda_{n+k} \equiv -\bs{1}+\lambda_{n+k}\in \N^-$ since each player can move to a position from which the opponent is forced to move to zero.  In this situation we see Left prefers $\bs{n}$ over $\bs{m}$, so again $\bs{n} \not \leqq \bs{m}$.
\end{proof}

Theorem \ref{intfuzzy} tells us that, modulo $\mathcal{E}$, the games $\{0,\bs{1} \v \cdot\}$ and $\{0 \v \cdot \}$ are distinguishable, as the option to $0$ does not in general dominate the option to $\bs{1}$.  Thus, in the dead-ending universe, there exist ends that are not integers.  However, if we restrict ourselves to the subuniverse of dead ends, then the ordering given in Theorem \ref{int_order} implies that every end reduces to an integer.  This fact is presented in the following lemma.

\begin{lem} \label{end_int}
If $G$ is a dead end then $G\equiv \bs{n}$, modulo the closure of dead ends, where $n = l(G)$ if $G$ is a right end and $n=-r(G)$ if $G$ is a left end.
\end{lem}

\begin{proof}
Let $G$ be a dead right end (the argument for left ends is symmetric).  Assume by induction that every option $G_i^L$ of $G$ (necessarily a dead right end) is equivalent to the integer $l(G_i^L)$.  Modulo dead ends, by Theorem \ref{int_order}, these left options are totally ordered; thus $G=\{G_1^L \v \cdot\}$ for $G_1^L$ with smallest left-length.  Then $G$ is the canonical form of the integer $l(G_1^L)+1 = l(G)$.
\end{proof}

Lemma \ref{end_int} shows that  the closure of dead ends has precisely the same monoid as the set of canonical-form integers.  The game of {\sc domineering} on $1\by n$ and $n \by 1$ strips is an instance of these universes.
The results of this section  allow us to completely describe the monoid, which we present  in Theorem \ref{int_monoid}.  By $\mathbb{N}$ we mean the set of natural numbers, including zero.

\begin{thm} \label{int_monoid}
Under the mapping $\bs{n} \mapsto \alpha^{n},$
the \mise monoid of the set of normal-play canonical-form integers is
$$\mathscr{M}_{\mathbb{Z}} = \langle 1,\alpha,\alpha^{-1} \v \alpha \cdot \alpha^{-1}=1\rangle \cong (\mathbb{Z},+),$$
with outcome partition
$$\N^-=\{0\}, \L^-=\{\alpha^{-n} \v n \in \mathbb{N}\}, \R^- = \{\alpha^{n} \v n \in \mathbb{N}\},$$
and total ordering
$$\alpha^n\gneqq \alpha^m \Leftrightarrow n<m.$$
\end{thm}

\section{Numbers}
\subsection{The monoid  of $\mathbb{Q}_2$.}
We say that a game $\bs{a}$ is a {\em non-integer number} in a universe $\mathcal{U}$ if it is equivalent, modulo $\mathcal{U}$, to the normal-play canonical form of a (non-integer) dyadic rational:
$$\bs{a}=\frac{\bs{m}}{\bs{2^j}}= \left \{ \frac{\bs{m}-\bs{1}}{\bs{2^j}} \bigg | \frac{\bs{m}+\bs{1}}{\bs{2^j}} \right \},$$
with $j>0$ and $m$ odd.    The set of all integer and non-integer (combinatorial game) numbers is thus the set of  {\em dyadic rationals}, which we denote by $\mathbb{Q}_2$.
As we did for integers in the previous section, we now determine the outcome of a general sum of dyadic rationals and thereby describe the \mise monoid of the closure of numbers.

Note that the sum of two non-integer numbers (even if both are positive) is not necessarily another number.  For example, in general \mise play, $\bs{1}+\bs{1}/\bs{2} = \{\bs{1}/\bs{2},\bs{1} \v \bs{2}\} \not = \bs{3}/\bs{2}$ implies that $\bs{1}/\bs{2}+\bs{1}/\bs{2} = \{\bs{1}/\bs{2} \v 1+\bs{1}/\bs{2} \} \not = \bs{1}$.   We will see that, unlike integers, the set of dyadic rationals is not closed under disjunctive sum even when restricted to the dead-ending universe; however, closure does occur when we restrict to numbers alone.  

Lemma \ref{rational_outcome} below --- analogous to  Lemma \ref{end_outcome} of the previous section --- shows that the outcome of a sum of numbers is determined by the left- and right-lengths of the individual numbers. To prove this, we require Lemma \ref{leftlengthR}, which establishes a relationship between the left- or right-lengths of numbers and their options; and to prove Lemma \ref{leftlengthR}, we need the following proposition.

\begin{prop} \label{propo}
If $a \in \mathbb{Q}_2 \setminus \mathbb{Z}$ then at least one of $\bs{a}^{RL}$ and $\bs{a}^{LR}$ exists, and either $\bs{a}^L=\bs{a}^{RL}$ or $\bs{a}^R=\bs{a}^{LR}$.
\end{prop}

\begin{proof}
Let $\bs{a}=\bs{m}/\bs{2^j}$ with $j>0$ and $m$ odd.  If $m\equiv 1$ (mod 4) then 
$$\bs{a}^L=\frac{\bs{m}-\bs{1}}{\bs{2^j}}, 
\;
\bs{a}^R=\frac{\bs{m}+\bs{1}}{\bs{2^j}} = \frac{\frac{\bs{m}+\bs{1}}{\bs{2}}}{\bs{2}^{\bs{j}-\bs{1}}}
 = \left \{ \frac{\frac{\bs{m}-\bs{1}}{\bs{2}}}{\bs{2}^{\bs{j}-\bs{1}}} \; \bigg |  \;
 \frac{\frac{\bs{m}+\bs{3}}{\bs{4}}}{\bs{2}^{\bs{j}-\bs{1}}} \right \},$$
so $\bs{a}^L=\bs{a}^{RL}$.  Otherwise,  $m\equiv 3$ (mod 4) and then
$$\bs{a}^L=\frac{\bs{m}-\bs{1}}{\bs{2^j}}
= \frac{\frac{\bs{m}-\bs{1}}{\bs{2}}}{\bs{2}^{\bs{j}-\bs{1}}} = \left \{ \frac{\frac{\bs{m}-\bs{3}}{\bs{2}}}{\bs{2}^{\bs{j}-\bs{1}}} \; \bigg |  \; \frac{\frac{\bs{m}+\bs{1}}{\bs{2}}}{\bs{2}^{\bs{j}-\bs{1}}} \right \},
\bs{a}^R=\frac{\bs{m}+\bs{1}}{\bs{2^j}},$$
so $\bs{a}^R=\bs{a}^{LR}$.
\end{proof}

Note that if $a>0$ is a dyadic rational then $l(\bs{a})=1+l(\bs{a}^L)$, and if $a<0$ is a dyadic rational then  $r(\bs{a})=1+r(\bs{a}^R)$.  We also have the following inequalities for left-lengths of right options and right-lengths of left options, when $a$ is a non-integer dyadic rational.
 
\begin{lem} \label{leftlengthR}
If $a \in \mathbb{Q}_2 \setminus \mathbb{Z}$ is positive  then $l(\bs{a}^R)\leq l(\bs{a})$; if $a$ is negative then $r(\bs{a}^L)\leq r(\bs{a})$.
\end{lem}

\begin{proof}
Assume $a>0$ (the argument for $a<0$ is symmetric).  Since $\bs{a}$ is in canonical form, both $\bs{a}^L$ and $\bs{a}^R$ are positive numbers.
If $\bs{a}^L= \bs{a}^{RL}$ then $l(\bs{a}^R)=1+l(\bs{a}^{RL})=1 + l(\bs{a}^L) = l(\bs{a})$.  Otherwise $\bs{a}^R=\bs{a}^{LR}$, by Proposition \ref{propo}; then $\bs{a}^L$ is not an integer because $\bs{a}^{LR}$ exists, so by induction we obtain $l(\bs{a}^R)=l(\bs{a}^{LR})\leq l(\bs{a}^L) = l(\bs{a})-1<l(\bs{a})$. 
\end{proof}

We can now determine the outcome of a general sum of numbers, both integer and non-integer.

\begin{lem} \label{rational_outcome}
If $\{a_i\}_{1 \leq i \leq n}$ and $\{b_i\}_{1 \leq i \leq m}$ are sets of positive and negative numbers, respectively, with $k= \sum_{i=1}^n l(\bs{a_i}) - \sum_{i=1}^m r(\bs{b_i})$, then
\begin{displaymath}
o^- \left (\sum_{i=1}^n \bs{a_i} + \sum_{i=1}^m \bs{b_i} \right ) = \left\{
     \begin{array}{lr}
       \L^- & \textrm{if } k <0 \\
       \N^- & \textrm{if } k =0 \\
       \R^- & \textrm{if } k >0.
     \end{array}
   \right.
\end{displaymath} 
\end{lem}

\begin{proof}
Let $G=\sum_{i=1}^n \bs{a_i} + \sum_{i=1}^m \bs{b_i}$. All followers of $G$ are also of this form, so assume the result holds for every proper follower of $G$.  Suppose $k<0$.  If $n=0$ then Left will run out of moves first because Left cannot move last in any negative number. So assume $n>0$.  Left moving first can move in an $\bs{a_i}$ to reduce $k$ by one (since $l(\bs{a_i}^L)=l(\bs{a_i})-1$), which is a left-win position by induction.  If Right moves first in an $\bs{a_i}$ then $k$ does not increase, since $l(\bs{a_i}^R)\leq l(\bs{a_i})$ by Lemma \ref{leftlengthR}, so the position is a left-win by induction; if Right moves first in a $\bs{b_i}$ then $k$ does increase by one, but Left can respond in an $\bs{a_i}$ (since $n>0$) to bring $k$ down again, leaving another left-win position, by induction.  Thus $G\in \L^-$ if $k<0$.

The argument for $k>0$ is symmetric.  If $k=0$ then either $G=0$ is trivially next-win, or both $n$ and $m$ are at least $1$ and both players have a good first move to change $k$ in their favour. 
\end{proof}

Lemma \ref{rational_outcome} shows that in general \mise play, the outcome of a sum of numbers is completely determined by the left-lengths and right-lengths of the positive and negative components, respectively.  From this we can conclude that, modulo the closure of canonical-form numbers, a positive number $\bs{a}$ is equivalent to every other number with left-length $l(\bs{a})$.  In particular, every positive number $\bs{a}$ is equivalent to the integer $\bs{l(a)}$.  This is Corollary \ref{modnumbers} below; together with Theorem \ref{rationalzero}, it will allow us to describe the monoid of canonical-form numbers.

\begin{cor} \label{modnumbers}
If $\bs{a}$ is a number, then
\begin{displaymath}
\bs{a} \equiv \left\{
     \begin{array}{lr}
       \bs{l(a)} & \textrm{if } a\geq 0, \\
       \bs{-r(a)}  & \textrm{if } a<0. 
     \end{array}
   \right.
\end{displaymath} 
\end{cor}

  As an example, the dyadic rational $\bs{1}/\bs{2}$ is equivalent to $\bs{l}(\bs{1}/\bs{2}) = \bs{1}$,  and $-\bs{3}/\bs{4} \equiv -\bs{r}(-\bs{3}/\bs{4} ) = \bs{-2}$, modulo $\mathbb{Q}_2$.
Note that these equivalencies do not hold in the larger universe of $\mathcal{E}$; indeed, as we  see in section 4.2, if $a\not = b$ are numbers then $a\not \equiv b$ (mod $\mathcal{E}$).

We see then that the closure of numbers is isomorphic to the closure of just integers; when restricted to numbers alone, every non-integer is equivalent to an integer.  Thus the \mise monoid of numbers, given below, is the same monoid presented in Theorem \ref{int_monoid}.  The partial order of the set of numbers, modulo $\mathcal{E}$, is described in Section 4.2.

\begin{thm} \label{rational_monoid}
Under the mapping $\bs{a} \mapsto \alpha^{n},$ where $n=l(G)$ if $a\geq0$ and $n=-r(G)$ if $a<0$,
the \mise monoid of the set of canonical-form dyadic rationals is
$$\mathscr{M}_{\mathbb{Q}_2} = \langle 1,\alpha,\alpha^{-1} \v \alpha \cdot \alpha^{-1}=1\rangle \cong (\mathbb{Z},+),$$
with outcome partition
$$\N^-=\{0\}, \L^-=\{\alpha^{-n} \v n \in \mathbb{N}\}, \R^- = \{\alpha^{n} \v n \in \mathbb{N}\}.$$
\end{thm}

As with integers, some of the structure found in the number universe is also present in the larger universe $\mathcal{E}$. We end this subsection with a proof that  all numbers --- not just integers --- are invertible in the universe of dead-ending games. We require the following lemma, an extension of Lemma \ref{rational_outcome}.

\begin{lem} \label{rational+X}
If $\{a_i\}_{1 \leq i \leq n}$ and $\{b_i\}_{1 \leq i \leq m}$ are sets of positive and negative numbers, respectively, and $\sum_{i=1}^n l(\bs{a_i}) - \sum_{i=1}^m r(\bs{b_i})<0$, then 
$$o^- \left (\sum_{i=1}^n \bs{a_i} + \sum_{i=1}^m {\bs{b_i}} +X \right ) = \L^-,$$
for any left end $X \in \mathcal{E}$.
\end{lem}

\begin{proof}
The argument from Theorem \ref{rational_outcome} works again, since if Right uses his turn to play in $X$ then Left responds with a move in $\bs{a_1}$ to decrease $k$ by 1, which is a win for Left by induction.
\end{proof}

\begin{thm} \label{rationalzero}
If $a\in \mathbb{Q}_2$ then $\bs{a}+\overline{\bs{a}}\equiv 0$ (mod $\mathcal{E}$).
\end{thm}

\begin{proof}   Without loss of generality we can assume $a$ is positive.
Since every follower of a number is also a number, we can use Lemma \ref{zerolemmaS}.  That is, it suffices to show $\bs{a}+\overline{\bs{a}}+X \in \L^- \cup \N^-$ for any left end $X\in \mathcal{E}$.  If $X=0$ this is true by Lemma \ref{rational_outcome}.  If $X\not= 0$ then we claim $\bs{a}+\overline{\bs{a}}+X \in \L^-$; assume this holds for all followers of $\bs{a}$. 
Left can win playing first on $\bs{a}+\overline{\bs{a}}+X$ by moving to $\bs{a}^L$, since $l(\bs{a}^L)-r(\overline{\bs{a}})=l(\bs{a}^L)-l(\bs{a})<0$ implies $\bs{a}^L+\overline{\bs{a}}+X \in \L^-$ by Lemma \ref{rational+X}.  If Right plays first in $X$ then again Left wins by moving $\bs{a}$ to $\bs{a}^L$; if Right plays first in $\overline{\bs{a}}$ then Left copies in $\bs{a}$ and wins on $\bs{a}^L + \overline{\bs{a}^L}+X \in \L^-$ by induction. \end{proof}

Theorem \ref{rationalzero} shows that in dead-ending games like {\sc domineering}, {\sc hackenbush}, etc., any position corresponding to a normal-play canonical-form number has an additive inverse under mis\`ere play.  So, for example, the  positions in Figure \ref{hack} would cancel each other in a game of \mise {\sc hackenbush}.

\begin{figure}[htp] 
\begin{center}
\includegraphics[scale=0.6]{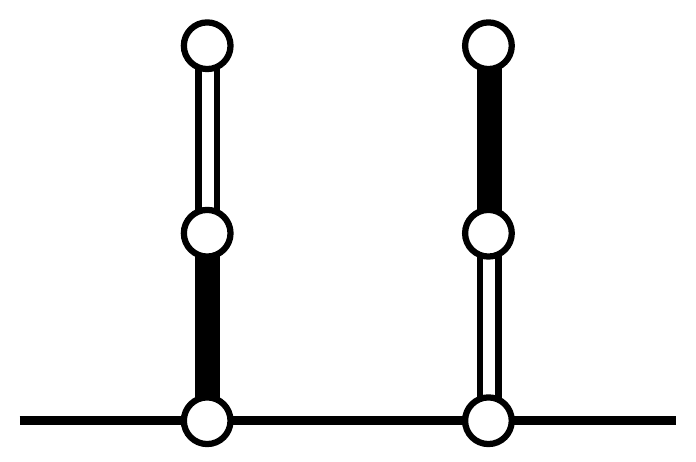}
\caption{Normal-play canonical forms of $1/2$ (left) and $-1/2$ (right) in {\sc hackenbush}.}
\label{hack}
\end{center}
\end{figure}

\subsection{The partial order of numbers inside $\mathcal{E}$.}
In section 3 we found that all integers were incomparable in the dead-ending universe.  We will see now that non-integer numbers are a bit more cooperative; although not totally ordered, we do have a nice characterization of the partial order of numbers in the universe $\mathcal{E}$.  First note that any two distinct numbers are distinguishable modulo $\mathcal{E}$; this is an immediate  corollary of the following theorem of \cite{DorbecRSS}, which extends a result of \cite{siege} referenced earlier.

\begin{thm} {\em \cite{DorbecRSS}} \label{DorbecRSSthm} If $G \geqq H$ (mod $\mathcal{E}$) then $G\geq H$ in normal play. 
\end{thm}

\begin{cor} \label{aneqb} If $a,b \in \mathbb{Q}_2$ and $a\not = b$ then $\bs{a}\not \equiv \bs{b}$ (mod $\mathcal{E}$).
 \end{cor}
 
Theorem \ref{DorbecRSSthm} says that if $\bs{a} \geqq \bs{b}$ (mod $\mathcal{E}$) then $a \geq b$ as real numbers (or as normal-play games).  The converse is clearly not true  for integers, by Theorem \ref{intfuzzy}; it is also not true for non-integers, since $\bs{1/2} - \bs{1/2}\in \N^-$ while $\bs{3/4}-\bs{1/2} \in \R^-$ (which the reader can verify), so that $\bs{1/2} \not \leqq \bs{3/4}$ (mod $\mathcal{E}$).  Theorem \ref{numorder} shows that the additional stipulation  $l(\bs{a}) \leq l(\bs{b})$ is sufficient for $\bs{a} \geqq \bs{b}$ (mod $\mathcal{E}$).  To prove this result we need the following lemmas.
As before, non-bolded symbols represent  actual numbers, so that `$a<b$' indicates inequality of $a$ and $b$ as rational numbers (or as normal-play games), and $a^L$ means the rational number corresponding to the left-option of the game $\bs{a}$ in canonical form.
  Recall that if $\bs{x} = \{\bs{x}^L | \bs{x}^R\}$ is in (normal-play) canonical form then $\bs{x}$ is the {\em simplest} number (i.e., the number with smallest birthday) such that $x^L <x< x^R$.  Thus, if $x^L<x,y<x^R$ and $x\not =y$ then $\bs{x}$ is nn $\bs{y}$.

\begin{lem} \label{orderlemma1}
If $a$ and $b$ are positive numbers such that \mbox{$a^L<b<a$}, then $l(\bs{a}^L)<l(\bs{b})$.
\end{lem}

\begin{proof}
We have $a^L<b<a<a^R$, so $\bs{a}$ must be simpler than $\bs{b}$.  Thus $b^L\geq a^L$, since otherwise $b^L<a^L<b<b^R$ would imply that $\bs{b}$ is simpler than $\bs{a^L}$, which is simpler than $\bs{a}$.  Now, if $b^L = a^L$ then $l(\bs{a}^L) = l(\bs{b}^L)= l(\bs{b})-1< l(\bs{b})$, and if $b^L>a^L$ then by induction $a^L<b^L<b<a$ gives $l(\bs{a}^L)<l(\bs{b}^L)=l(\bs{b})-1<l(\bs{b})$. 
\end{proof}

Lemma \ref{orderlemma1} is used to prove Lemma \ref{orderlemma2} below, which is needed for the proof of Theorem \ref{numorder}. Note that in the following two arguments we frequently use the fact that, if $a\geqq b$ (mod $\mathcal{E}$), then Left wins on the position $a+\overline{b}+X$ whenever she wins $X\in \mathcal{E}$.

\begin{lem} \label{orderlemma2}
If $a$ and $b$ are positive  numbers such that $a^L<b< a$, then $\bs{a} \gneqq \bs{b}$ (mod $\mathcal{E}$).
\end{lem}

\begin{proof}
Note that $b\not \in \mathbb{Z}$ since there are no integers between $a^L$ and $a$ if $\bs{a}$ is in canonical form.
We must show that Left wins $\bs{a}+\overline{\bs{b}}+X$ whenever she wins $X\in \mathcal{E}$.
  \\
\\{\em Case 1:} $b^R=a$.
\\ Left can win $\bs{a}+\overline{\bs{b}}+X$ by playing her winning strategy on $X$.   If Right moves in $\bs{a}+\overline{\bs{b}}$ to $\bs{a}^R+\overline{\bs{b}}+X'$, then Left responds to $\bs{a}^R+\overline{\bs{b}^R}+X' = \bs{a}^R+\overline{\bs{a}}+X'$, which she wins by induction since $a^{RL}\leq a^L$ (see Proposition \ref{propo}) gives $a^{RL}<a<a^R$. 
If Right moves to $\bs{a}+\overline{\bs{b}}^R+X' = \bs{b}^R+\overline{\bs{b}}^R+X'$, with $X'\in \L^-\cup \P^-$ (since Left is playing her winning strategy in $X$), then Left's response depends on whether $\bs{b}^{RL}=\bs{b}^L$ or $\bs{b}^{LR}=\bs{b}^R$: if the former, Left moves to $\bs{b}^{RL}+\overline{\bs{b}}^R+X' = \bs{b}^{L}+\overline{\bs{b}^L}+X'\equiv X'$ (mod $\mathcal{E}$); if the latter then Left moves to
$\bs{b}^R+\overline{\bs{b}^L}^L+X' = \bs{b}^{R}+\overline{\bs{b}^{LR}}+X' = \bs{b}^{R}+\overline{\bs{b}^{R}}+X' \equiv X'$.  In either case Left wins as the previous player on $X' \in \L^- \cup \P^-$.

When Left runs out of moves in $X$, she moves to $\bs{a}^L+\overline{\bs{b}}+X$. By Lemma \ref{orderlemma1} we know $l(\bs{a^L})<l(\bs{b})$, and this gives $\bs{a}^L+\overline{\bs{b}}+X\in \L^-$ by Lemma \ref{rational+X}.
\\ \ \\{\em Case 2:} $b^R \not = a$.
\\ Note that $b^R$ cannot be greater than $a$, since $a^L<b<a<a^R$ implies $\bs{a}$ is simpler than $\bs{b}$, while $b^L<b<a<b^R$ would imply that $\bs{b}$ is simpler than $\bs{a}$.  So $b^R<a$, and together with $a^L<b<b^R$ this gives $a^L<b^R<a$, which shows $\bs{a}\gneqq \bs{b}^R$ (mod $\mathcal{E}$) by induction.  Similarly $b^{RL}\leq b^L < b <b^R$ implies $\bs{b}^R\gneqq \bs{b}$ (mod $\mathcal{E}$), by Case 1.  Then by transitivity we have $\bs{a}\gneqq \bs{b}$ (mod $\mathcal{E}$).

\end{proof}

With Lemma \ref{orderlemma2}  we can now prove Theorem \ref{numorder} below.  The symmetric result for negative numbers also holds.   

\begin{thm} \label{numorder}
If $a$ and $b$ are positive numbers such that $a> b$ and $l(\bs{a}) \leq l(\bs{b})$, then $\bs{a} \gneqq \bs{b}$ (mod $\mathcal{E}$).  
\end{thm}

\begin{proof}  By Corollary \ref{aneqb} we have $\bs{a}\not \equiv \bs{b}$ (mod $\mathcal{E}$), and so it suffices to show $\bs{a} \geqq \bs{b}$ (mod $\mathcal{E}$).  Again we have $b\not \in \mathbb{Z}$.
Since $a>b$, if $b>a^L$ then Lemma \ref{orderlemma2} gives $\bs{a} \gneqq \bs{b}$ (mod $\mathcal{E}$) as required.  So assume $b\leq a^L$.  Again let $X\in \mathcal{E}$ be a game which Left wins playing first; we must show Left wins $\bs{a}+\overline{\bs{b}}+X$ playing first.  Left should follow her winning strategy from $X$.  If Right plays to $\bs{a} + \overline{\bs{b}^L}+X'$, where $X'\in \L^-\cup \P^-$, then Left responds with 
$\bs{a}^L+\overline{\bs{b}^L}+X'$, which she wins by induction: $b^L <b \leq a^L$ and $l(\bs{b}^L)=l(\bs{b})-1\geq l(\bs{a})-1=l(\bs{a}^L)$ implies $\bs{a}^L \gneqq \bs{b}^L$ (mod $\mathcal{E}$).

If Right plays to $\bs{a}^R + \overline{\bs{b}}+X'$ (assuming this move exists --- that is, assuming $a \not \in \mathbb{Z}$) then Left's response is $\bs{a}^{RL} + \overline{\bs{b}}+X'$, if $a^{RL}>b$, or  $\bs{a}^R + \overline{\bs{b}^R}+X'$ if $a^{RL}\leq b$.  In the first case Left wins by induction because
$a^{RL}>b$ and $l(\bs{a}^{RL})=l(\bs{a}^R)-1 \leq l(\bs{a})-1 < l(\bs{b})$ implies $\bs{a}^{RL}\gneqq \bs{b}$ (mod $\mathcal{E}$).
In the latter case, note firstly that in fact $a^{RL}\not = b$, since we have already seen that as games they have different left-lengths.  Then we see $a^{RL}<b<a<a^R<a^{RR}$, which shows $\bs{a}^R$ must be simpler than $\bs{b}$.  This gives $b^R\leq a^R$, as otherwise $b^{L}<b<a<a^R<b^R$ would imply that $\bs{b}$ is simpler than $\bs{a}^R$.   If $b^R=a^R$ then $\bs{b}^R=\bs{a}^R$, and if $b^R<a^R$ then we can apply Lemma \ref{orderlemma2} to conclude that  $\bs{a}^R\gneqq \bs{b}^R$ (mod $\mathcal{E}$). In either case, Left wins $\bs{a}^R + \overline{\bs{b}^R}+X'$, with $X'\in\L^-\cup \P^-$, as the second player.

Finally, if Left runs out of moves in $X$ then she moves to $\bs{a}^L+\overline{\bs{b}}+X''$ where $X''$ is a dead left end; then Left wins by Lemma \ref{rational+X} because $l(\bs{a}^L)<l(\bs{a})\leq l(\bs{b}) = r(\overline{\bs{b}})$.
\end{proof}

\begin{cor} \label{partialordercor} For positive numbers $a,b\in \mathbb{Q}_2$, $\bs{a} \gneqq \bs{b}$ (mod $\mathcal{E}$) if and only if $a>b$ and $l(\bs{a})\leq l(\bs{b})$.\end{cor}

\begin{proof}
We need only prove the converse of Theorem \ref{numorder}.  Suppose $a>b$ and $l(\bs{a})>l(\bs{b})$; then by \cite{DorbecRSS} it cannot be that $\bs{a} \leqq \bs{b}$ (mod $\mathcal{E}$), so we need only show $\bs{a}\not \geqq \bs{b}$ (mod $\mathcal{E}$). 
We have $\bs{b}+\overline{\bs{b}}\in \N^-$, while $\bs{a}+\overline{\bs{b}}\in \R^-$, since in isolation the latter sum is equivalent to the positive integer $l(\bs{a})-l(\bs{b})$, by Theorem \ref{rational_monoid}. Thus $\bs{a}\not \geqq \bs{b}$ (mod $\mathcal{E}$).

\end{proof}

To completely describe the partial order of numbers within $\mathcal{E}$, it remains to consider the comparability of $\bs{a}$ and $\bs{b}$ when $a\geq 0$ and $b<0$ (or, symmetrically, when $a>0$ and $b\leq 0$).  As before we cannot have $\bs{a}\leqq \bs{b}$ (mod $\mathcal{E}$), and  the same argument as above ($\bs{b}+\overline{\bs{b}} \in \N^-$ and $\bs{a}+\overline{\bs{b}}\in \R^-$) shows $\bs{a} \not \geqq \bs{b}$ (mod $\mathcal{E}$)
.  The results of this section are summarized below.

\begin{thm}
The partial order of $\mathbb{Q}_2$, modulo $\mathcal{E}$, is given by
\begin{center}
\begin{tabular}{rl}
$\bs{a} \equiv \bs{b} \;($mod $\mathcal{E})$ & if $a=b$;\\
$\bs{a}\gneqq \bs{b} \;($mod $\mathcal{E})$ & if $0<a<b$ and $l(\bs{a})\leq l(\bs{b}),$\\
& or $b<a<0$ and $r(\bs{b})\leq r(\bs{a})$;\\
$\bs{a} || \bs{b}  \;($mod $\mathcal{E})$ & otherwise.\\
\end{tabular}
\end{center}
\end{thm}

\section{Zeros in the dead-ending universe}
We have found that integer and non-integer numbers, as well as all ends, satisfy $G+\overline{G}\equiv 0$ (mod $\mathcal{E}$).  It is not the case that every game in $\mathcal{E}$ has an additive inverse; for example, $*+*\not \equiv 0$ (mod $\mathcal{E})$, although the equivalence does hold in the dicot universe $\mathcal{D} \subset \mathcal{E}$.  Likewise, many familiar `all-small' games from normal-play, which have inverses in the dicot universe\footnote{The games $\uparrow = \{0 \v *\}$, $\downarrow = \{* \v 0\}$, and all other day-2 \mise dicots with the exception of $*2 = \{0,* \v 0,*\}$, are shown to be invertible modulo $\mathcal{D}$ in the first author's thesis (in progress),  using Lemma \ref{zerolemmaS}.}, are not invertible here.  

The following lemma describes an infinite family of games that are not invertible in the universe of dead-ending games.  

\begin{lem}
If $G=\{\bs{n_1},\ldots,\bs{n_k} \v \overline{\bs{n_1}}, \ldots, \overline{\bs{n_k}}\}$  with each $n_i\in \mathbb{N}$, then $G+\Gb \not \equiv 0$ (mod $\mathcal{E}$).
\end{lem}

\begin{proof}
Let $X=\{\bs{n_1},\ldots, \bs{n_k}\v \cdot \} \in \R^-$. Note that $G=\Gb$. We describe a winning strategy for Left playing second in the game $G+\Gb+X= G+G+X$.  Right has no first move in $X$, so Right's move is of the form $G+\overline{\bs{n_i}}+X$.  Left can  respond by moving $X$ to $\bs{n_i}$, leaving $G+0$.  Now Right plays in $G$ to a nonpositive integer, which as a right end must be in $\L^-$ or $\N^-$. 
\end{proof}

We conclude with an infinite family of games that are equivalent to zero in the dead-ending universe, which are not all of the form $G+\Gb$ for some $G$.  These games are illustrated in Figure \ref{zerofamily}.

\begin{thm} 
If $G$ is a dead-ending game such that every $G^L$ is a left end with an option to zero and every $G^R$ is a right end with an option to zero, then $G\equiv 0$ (mod ${\mathcal{E}}$).
\end{thm}

\begin{proof}
Let $X$ be any game in $\mathcal{E}$ and suppose Left wins $X$.  Then Left wins $G+X$ by following her strategy in $X$.  If Right plays in $G$ then he moves to $G^R+X'$ from a position $G+X'$ with $X'\in \L^-\cup \P^-$; Left can respond to $0+X'$ and win as the second player.  If both players ignore $G$ then eventually Left runs out of moves in $X$ and plays to $G^L+X''$, where $X''$ is a left end.  But $G^L$ is also a left end, so the sum is  a left-win by Lemma \ref{leftendL}.
\end{proof}

\begin{figure}[htp] 
\begin{center}
\includegraphics[height=4cm]{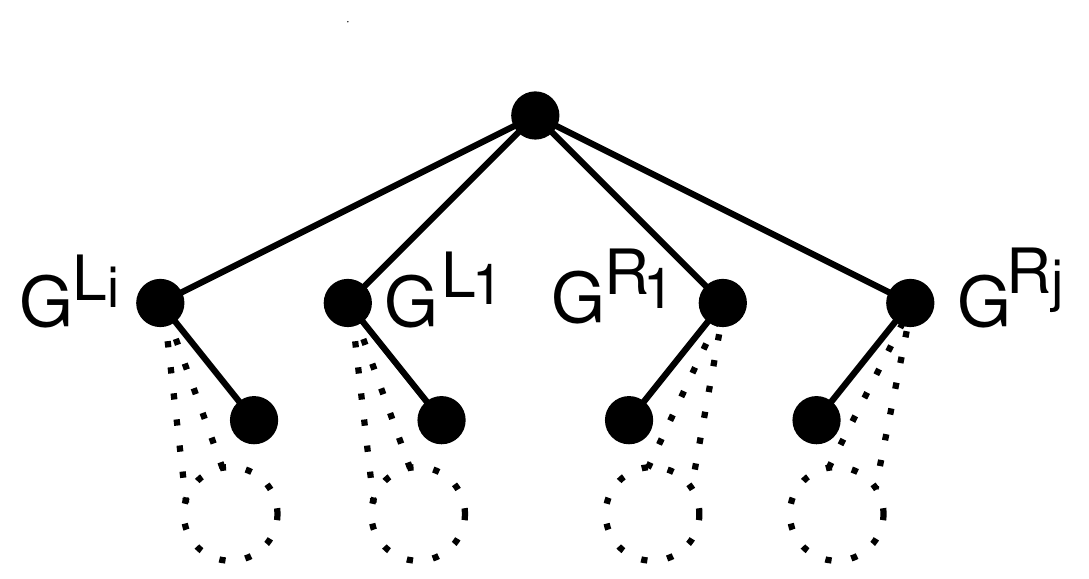}
\caption{An infinite family of games equivalent to zero modulo $\mathcal{E}$ ($i\geq 1, j\geq 1$).}
\label{zerofamily}
\end{center}
\end{figure}

\section{Future directions}
From this first initial investigation, the universe of dead-ending games appears to be filled with  potential for successful \mise analysis.  It includes as subuniverses many of the games that have already excited interest among combinatorial game theorists;  we hope some of the techniques of the present paper can be fruitfully applied to these subuniverses.

A natural extension of this work would include analysis of specific games, such as {\sc nogo, col, snort}, etc., in the context of the dead-ending universe.  It would also be interesting to consider other properties, besides dead-ending, of the `placement games' described in our opening paragraph.  

\section{Acknowledgments}
The authors greatly appreciate the comments and suggestions provided by Paul Dorbec, Neil McKay, Richard Nowakowski, Aaron Siegel, and \'{E}ric Sopena.

\end{document}